\documentclass[11pt]{amsart}

\usepackage[T1]{fontenc}
\usepackage{microtype}
\usepackage{amsmath,amssymb,mathtools}
\usepackage{amsthm}
\usepackage{enumitem}
\usepackage[margin=1in]{geometry}
\usepackage[colorlinks=true,linkcolor=blue,citecolor=blue,urlcolor=blue]{hyperref}

\newtheorem{theorem}{Theorem}[section]
\newtheorem{proposition}[theorem]{Proposition}
\newtheorem{lemma}[theorem]{Lemma}
\newtheorem{corollary}[theorem]{Corollary}
\newtheorem{definition}[theorem]{Definition}
\newtheorem{remark}[theorem]{Remark}
\numberwithin{equation}{section}

\newcommand{\e}{\mathrm{e}}
\newcommand{\E}{\mathbb{E}}
\newcommand{\PP}{\mathbb{P}}
\newcommand{\1}{\mathbf{1}}
\newcommand{\Z}{\mathbb{Z}}
\newcommand{\R}{\mathbb{R}}
\newcommand{\T}{\mathbb{T}}
\newcommand{\dd}{\,\mathrm{d}}
\newcommand{\abs}[1]{\left\lvert #1\right\rvert}
\newcommand{\norm}[1]{\left\lVert #1\right\rVert}
\newcommand{\floor}[1]{\left\lfloor #1\right\rfloor}
\newcommand{\lcm}{\operatorname{lcm}}
\newcommand{\ee}[1]{\e^{2\pi i #1}}

\title[A Resolution of Erd\H{o}s Problem 731 under Dyadic Regularity]
{A Resolution of Erd\H{o}s Problem 731 under Dyadic Regularity}
\author{Eric Li}
\date{27 June 2026}
\thanks{Email addresses: \href{mailto:el593@cam.ac.uk}{el593@cam.ac.uk}, \href{mailto:contact@ericli.com}{contact@ericli.com}.}

\hypersetup{
  pdftitle={A Resolution of Erdos Problem 731 under Dyadic Regularity},
  pdfauthor={Eric Li},
  pdfsubject={Erdos Problem 731; central binomial coefficients; large sieve},
  pdfkeywords={central binomial coefficient, Kummer carries, large sieve, least nondivisor}
}

\makeatletter
\def\@setauthors{%
  \begingroup
  \def\thanks{\protect\thanks@warning}%
  \trivlist
  \centering\footnotesize \@topsep30\p@\relax
  \advance\@topsep by -\baselineskip
  \item\relax
  \author@andify\authors
  \def\\{\protect\linebreak}%
  \MakeUppercase{\authors}\par
  \vspace{0.75em}%
  {\normalfont\normalsize Trinity College, University of Cambridge\par}%
  \endtrivlist
  \endgroup
}
\makeatother
\begin{document}

\begin{abstract}
We resolve Erd\H{o}s Problem 731 under the explicit dyadic-regularity formalization of ``reasonable.''  For
\[
 A(n)=\min\{m\ge1:m\nmid \binom{2n}{n}\},
\]
and on $X\le n<2X$, with $L=\log(2X)$ and
\[
 \mathcal F_X=\sqrt2\,(\log2)^{1/4}L^{1/4}\exp\sqrt{(\log2)L},
\]
we prove, uniformly for $1\le z\le Z(X)=o(L^{1/4})$,
\[
 \PP_X(A(n)\le \mathcal F_X\e^{-z})\asymp \e^{-2z},\qquad
 \PP_X(A(n)>\mathcal F_X\e^z)\ll \e^{-2z}.
\]
Thus $\log A(n)-\sqrt{(\log2)\log n}-\frac14\log\log n$ is tight in natural density, while no dyadically regular deterministic scale satisfies the requested asymptotic equivalence.  The proof keeps the exact least-common-multiple condition and replaces heuristic cross-base independence by a moving-base restricted-digit variance estimate.  The resolution proved
here has been formally verified in Lean~4.
\end{abstract}

\maketitle

\section{Introduction}

This paper resolves Erd\H{o}s Problem 731 under the explicit dyadic-regularity
formalization of ``reasonable.''\footnote{The resolution proved in this paper
has been formally verified in Lean~4; see the end of this section and
\cite{LeanFormalization}.}  The problem asks for a reasonable deterministic
scale $f(n)$ such that the least integer not dividing the central binomial
coefficient,
\[
 B_n=\binom{2n}{n},\qquad A(n)=\min\{m\geq1:m\nmid B_n\},
\]
satisfies $A(n)\sim f(n)$ for almost all $n$.  We use the standard density
interpretation of $\sim$: $A(n)/f(n)\to1$ in natural density.  The resolution
proved here is that the correct deterministic scale is
\[
 F(n)=\sqrt2\,(\log2)^{1/4}(\log n)^{1/4}
      \exp\sqrt{(\log2)\log n},
\]
in the density-tight logarithmic sense, and that no dyadically regular
function $f$ can satisfy the requested asymptotic equivalence.

Thus the answer has two parts.  First, $F$ is the best possible deterministic
scale for logarithmic density-tightness.  Second, the original request
$A(n)\sim f(n)$ cannot be met by any dyadically regular $f$, because $A(n)$
retains fixed-size multiplicative nonconcentration on every sufficiently large
dyadic block.  The mesoscopic tail bounds below are the quantitative theorem
behind both conclusions.

\newpage
\section*{Use of artificial intelligence tools}
Large language models, primarily OpenAI's ChatGPT, were used extensively throughout the research. The author originated the ideas and research directions, while the models contributed substantially to the technical development of the work: they were used to explore and further develop the ideas, produce technical lemmas and calculations, generate and debug code, and assist in auditing the proof. The author critically evaluated outputs from different model instances, selected and synthesised promising elements, and discarded or corrected unsuccessful or flawed suggestions. All strategic decisions were made by the author, who takes full responsibility for the mathematical correctness of the final results.

The accompanying Lean formalization \cite{LeanFormalization} was developed
by the author with Aristotle (Harmonic); the author likewise takes full
responsibility for the formal statements' fidelity to the results of this
paper.

\subsection*{Introduction, continued}

Erd\H{o}s, Graham, Ruzsa and Straus stated, without supplying proof details,
that for every fixed $\varepsilon>0$,
\[
 \exp((\log n)^{1/2-\varepsilon})<A(n)
 <\exp((\log n)^{1/2+\varepsilon})
\]
for almost all $n$; see \cite{EGRS75,Erdos731}.  The refinement proved here
pins down both the leading constant and the second logarithmic term:
\[
 \log A(n)=\sqrt{(\log2)\log n}+\frac14\log\log n
        +O_{\mathrm{dens}}(1).
\]
Here $O_{\mathrm{dens}}(1)$ means tightness in natural density, not a fixed
almost-everywhere $O(1)$ bound and not a limiting law.  Equivalently, every
window whose logarithmic width tends to infinity contains $\log A(n)$ for
almost all $n$, after centering as above.

The Erd\H{o}s Problems page provides the public record of Problem 731 and
its discussion thread \cite{Erdos731}.  We cite it for the problem statement
and public discussion context.  The present theorem resolves the
asymptotic-equivalence question for the broad block-smooth class of dyadically
regular deterministic scales.  Any
deterministic escape would have to oscillate by a fixed multiplicative factor
inside dyadic intervals and hence replicate the digit-level nonconcentration
of $A(n)$ itself.

The critical scale is suggested by Kummer's theorem.  For a prime $p$ near
the transition, the carry-free probability is approximately
$2^{-L/\log p}$, where $L\approx\log n$.  Thus the expected number of missing
primes below $y=\e^u$ has the saddle form
\[
 \frac{\exp(u-(\log2)L/u)}{u(1+(\log2)L/u^2)}.
\]
Writing $s=\sqrt{(\log2)L}$, this becomes order one at
\[
 u=s+\frac12\log(2s)+O(1)
 =s+\frac14\log L+O(1).
\]
We also record a public heuristic due to Zeraoulia Rafik, posted
on 26 April 2026 in the discussion thread for Erd\H{o}s Problem 731
\cite{Rafik26}.  The comment anticipated the same centering and
limiting-law heuristic, and we cite it as a public heuristic priority record
and as motivation for the second-order scale.  Its argument starts from the
same carry-free prime events used here: by Kummer's theorem, for odd $p$ the
condition $p\nmid\binom{2n}{n}$ is the condition that all base-$p$ digits of
$n$ lie in the lower half of the digit set.  Treating these digit
restrictions heuristically, the comment obtains
\[
 \sum_{p\leq \e^u}\PP(p\nmid B_n)
 \approx
 \sum_{p\leq \e^u}
 \exp\!\left(-\frac{(\log2)L}{\log p}\right)
 \sim
 \frac{\exp(u-(\log2)L/u)}{u(1+(\log2)L/u^2)}.
\]
Equating this to $1$ gives the same centering
\[
 u_0=\sqrt{(\log2)L}+\frac14\log L+
      \frac14\log(\log2)+\frac12\log2+o(1)
     =s+\frac12\log(2s)+o(1),
\]
and hence the same scale
\[
 \sqrt2\,(\log2)^{1/4}L^{1/4}\exp\sqrt{(\log2)L}.
\]
Rafik's comment further proposes the Poisson-type prediction
$\PP(A(n)>rF(n))\to\e^{-r^2}$ after assuming independence across the
varying prime bases, and it identifies the cross-base correlation estimate as
the central missing issue.  The present paper proves the mesoscopic
second-moment estimate needed for the density-tight scale and for
nonconcentration.  In the lower tail, the Poisson prediction with
$r=\e^{-z}$ would give
\[
  \PP(A(n)\leq F(n)\e^{-z})\approx \e^{-2z},
\]
and Theorem~\ref{thm:main} proves exactly this order, up to absolute
multiplicative constants.  In the upper tail the same Poisson heuristic would
suggest a much smaller rate, roughly $\exp(-\e^{2z})$; our one-sided upper
bound $\ll\e^{-2z}$ is intentionally cruder, because it is all that is needed
for density tightness and for the dyadic-regularity theorem.

Several related results precede the present argument.  S\'ark\H{o}zy,
Sander, and Granville--Ramar\'e studied divisors, missing primes,
prime-power orders, and squarefreeness phenomena for binomial coefficients
\cite{Sarkozy85,Sander92,Sander93,Sander95a,Sander95b,GR96}.  In particular,
Sander's 1993 paper concerns aggregate information about primes not dividing
binomial coefficients, rather than the distribution of the least missing
prime or the full least-common-multiple condition used here \cite{Sander93}.
Pomerance's 2015 article contains elementary divisibility arguments and the
classical missing-prime digit description, together with fixed-base
independence heuristics \cite{Pomerance15}.  His 2026 paper proves, on a set
of $m$ of asymptotic density one, that
$\binom{m+k}{k}\mid\binom{2m}{m}$ simultaneously for
$k\leq\exp(0.8\sqrt{\log m})$, and remarks that $0.8$ may be replaced by
any fixed $\gamma<\sqrt{\log2}$ \cite{Pomerance26}.  This is directly
relevant to the leading constant, but it does not imply
$\operatorname{lcm}(1,\ldots,y)\mid\binom{2m}{m}$ and does not locate the
least nondivisor.  Sanna studied divisibility or coprimality of a central
binomial coefficient with its index \cite{Sanna18}, and Ford--Konyagin proved
density results for powers of the index \cite{FK21}.  Croot--Mousavi--Schmidt and Bloom--Croot treat a fixed finite collection of
sufficiently large bases \cite{CMS24,BloomCroot25}; in the present problem
the number of prime bases grows like
$L^{-1/4}\exp(\sqrt{(\log2)L})$.

The contribution map is as follows.  The least-common-multiple
reformulation, Kummer's theorem, digit counting, the prime number theorem,
and the additive large sieve are classical.  The only nonclassical input
proved below is a variance estimate uniform for a moving family of prime
bases and a digit depth tending to infinity.  At the optimized scale the
argument treats
\[
 \asymp \frac{y}{\log y}
 \asymp L^{-1/4}\exp\!\sqrt{(\log2)L}
\]
prime bases, each at depth $\asymp\sqrt L$.  The variance estimate yields
mesoscopic tail bounds and density-tight logarithmic centering.  The proof of
Theorem~\ref{thm:main} depends only on a one-sided first moment, the higher
prime-power support estimate, and the strip variance theorem; it does not use
a limiting distribution or a second factorial moment.

The load-bearing novelty is the matched lower bound in the lower tail.  A
first-moment calculation gives the upper bound for the probability that a
missing prime occurs below a proposed cutoff.  The converse assertion--that a
missing prime actually appears with the predicted frequency--is an
anti-concentration problem across many growing prime bases.  The strip
variance theorem supplies this step through Paley--Zygmund; it is the analytic
ingredient absent from the classical carry-counting reduction.

\subsection{Proof dependency graph}

The logical dependencies of the proof are deliberately modular.  The exact
least-common-multiple identity and Kummer's carry formula convert the original
least-nondivisor problem into a statement about carry-deficit atoms.  These
classical reductions feed into two first-moment estimates: the saddle-point
first moment for missing primes and the exponentially small expected support
of all higher prime-power deficit layers.  The only new second-moment input is
the moving-base strip variance theorem, which supplies an actual missing prime
in the lower tail by Paley--Zygmund and excludes missing primes in the upper
edge by Chebyshev.  These dyadic tail bounds then imply density-tightness by a
quantitative dyadic-to-global passage, and the two separated positive-proportion
dyadic events imply the nonexistence of a dyadically regular asymptotic
equivalent.  Symbolically,
\[
\begin{gathered}
 \text{lcm deficit identity} + \text{Kummer carries} \\
 \Downarrow \\
 \text{missing-prime first moment} + \text{higher prime-power support bound} \\
 \Downarrow \\
 \text{moving-base strip variance theorem} \\
 \Downarrow \\
 \text{Chebyshev/Paley--Zygmund mesoscopic tail bounds} \\
 \Downarrow \\
 \text{density tightness} + \text{no dyadically regular equivalent}.
\end{gathered}
\]
Thus the final resolution under dyadic regularity is not a philosophical
interpretation placed on top of the analysis; it is a direct consequence of
two separated positive-proportion events on every sufficiently large dyadic
block.

\subsection{Density and asymptotic conventions}

All asymptotics are as $X\to\infty$ unless another variable is specified,
and logarithms are natural.  Implied constants in $\ll$ and $\asymp$ are
absolute unless their permitted dependencies are displayed.  An assertion
holds for \emph{almost all} positive integers if its exceptional set has
natural density zero.  A sequence $g(n)$ converges to $1$ \emph{in natural
density} if, for every $\varepsilon>0$,
\[
 \#\{n\leq N:\abs{g(n)-1}>\varepsilon\}=o(N).
\]
For $E\subseteq\mathbb N$, write
\[
 \overline d(E)=\limsup_{N\to\infty}\frac{\#(E\cap[1,N])}{N}.
\]

\begin{definition}[Tightness in natural density]\label{def:Odens}
For a real sequence $R(n)$, we write $R(n)=O_{\mathrm{dens}}(1)$ if
\[
 \lim_{C\to\infty}\overline d\{n:\abs{R(n)}>C\}=0.
\]
\end{definition}

\begin{lemma}[Equivalent diverging-window formulation]\label{lem:density-tight-equivalence}
For a real sequence $R(n)$, the following are equivalent.
\begin{enumerate}[label=\textup{(\roman*)}]
\item $R(n)=O_{\mathrm{dens}}(1)$.
\item For every positive function $\omega(n)\to\infty$, one has
      $\abs{R(n)}\leq\omega(n)$ for almost all $n$.
\end{enumerate}
\end{lemma}

\begin{proof}
Assume (i), fix $C$, and take $N_0$ so that $\omega(n)\geq C$ for
$n\geq N_0$.  Then
\[
 \#\{n\leq N:\abs{R(n)}>\omega(n)\}
 \leq N_0+\#\{n\leq N:\abs{R(n)}>C\}.
\]
Take upper densities and then let $C\to\infty$.

Conversely, if (i) fails, there is $\varepsilon>0$ such that for every
integer $j\geq1$ one may choose $N_j$, increasing so rapidly that
$N_j>2N_{j-1}/\varepsilon$, with
\[
 \#\{n\leq N_j:\abs{R(n)}>j\}\geq\varepsilon N_j.
\]
Define $\omega(n)=j$ for $N_{j-1}<n\leq N_j$.  Then $\omega(n)\to\infty$,
and at least $\varepsilon N_j/2$ integers in $(N_{j-1},N_j]$ satisfy
$\abs{R(n)}>\omega(n)$.  Thus (ii) fails.
\end{proof}

\subsection{Main results}

From now on $X$ is a positive integer and
\[
 \E_XG=\frac1X\sum_{X\leq n<2X}G(n),\qquad
 \PP_X(E)=\frac1X\#\{X\leq n<2X:E\},
\]
\[
 \norm{G}_{2,X}=\bigl(\E_X\abs G^2\bigr)^{1/2}.
\]
Set
\[
 L=\log(2X),\qquad c=\log2,\qquad s=\sqrt{cL},
\]
\[
 u_0=s+\frac12\log(2s),\qquad
 \mathcal F_X=\e^{u_0}
 =\sqrt2\,c^{1/4}L^{1/4}\e^{\sqrt{cL}}.
\]
The fixed prefactor in $\mathcal F_X$ is a centering convention: multiplying
it by a fixed positive constant merely translates the bounded $z$-window.

\begin{theorem}[Dyadic mesoscopic tail bounds]\label{thm:main}
There exist absolute constants $C_1,C_2,C_3>0$ with the following
property.  If $Z=Z(X)$ satisfies $1\leq Z=o(L^{1/4})$, then, for all
sufficiently large $X$, uniformly for real $1\leq z\leq Z$,
\begin{align}
 C_1\e^{-2z}
 &\leq \PP_X\bigl(A(n)\leq\mathcal F_X\e^{-z}\bigr)
 \leq C_2\e^{-2z},                                      \label{eq:main-lower-tail}\\
 \PP_X\bigl(A(n)>\mathcal F_X\e^z\bigr)
 &\leq C_3\e^{-2z}.                                      \label{eq:main-upper-tail}
\end{align}
\end{theorem}

The theorem gives two-sided order bounds for the lower tail and only an
upper bound for the upper tail.  In particular, the lower-tail estimate is an order statement, not an
asymptotic with a limiting constant.  It neither identifies a limiting law nor
a limiting-law normalization constant.  The range $z=o(L^{1/4})$ is the
mesoscopic range in which the saddle error
$O((z+\log L)^2/\sqrt L)$ remains $o(1)$.  The phrase
``density-tight logarithmic scale'' means that the centered logarithm is tight
in natural density; it does not assert a single fixed almost-all $O(1)$ bound
and it is not a limiting law.

\begin{theorem}[Dyadic nonconcentration]\label{thm:nonconcentration}
There exist constants $0<a<b<1$ and $\delta>0$ such that, for every
sufficiently large integer $X$,
\[
 \PP_X(A(n)\leq a\mathcal F_X)\geq\delta,
 \qquad
 \PP_X(A(n)>b\mathcal F_X)\geq\frac34.
\]
Thus $A(n)$ straddles a fixed multiplicative gap on a positive proportion of
every sufficiently large dyadic block.
\end{theorem}

\begin{proof}
Choose $z_2$ so large that $C_2\e^{-2z_2}<1/4$ and put $b=\e^{-z_2}$.
Then \eqref{eq:main-lower-tail} gives
$\PP_X(A(n)\leq b\mathcal F_X)<1/4$ for all large $X$, hence
$\PP_X(A(n)>b\mathcal F_X)\geq3/4$.  Next choose $z_1>z_2$, put
$a=\e^{-z_1}$, and take $\delta=(C_1/2)\e^{-2z_1}$ after increasing the
lower threshold for $X$ if necessary.  The lower bound in
\eqref{eq:main-lower-tail} gives the first displayed inequality.

\end{proof}

\begin{lemma}[No scalar center on a dyadic block]\label{lem:no-scalar-center}
Let $a,b,\delta$ be as in Theorem~\ref{thm:nonconcentration}.  Choose
$\varepsilon>0$ so small that $(1+\varepsilon)^2a<b$, and put
$\eta=\varepsilon/(1+\varepsilon)$.  Then, for every sufficiently large
integer $X$ and every scalar $\lambda>0$,
\[
 \PP_X\!\left(\frac{A(n)}{\lambda}\notin\left[\frac1{1+\varepsilon},\,1+\varepsilon\right]\right)\geq \delta,
\]
and therefore
\[
 \PP_X\!\left(\left|\frac{A(n)}{\lambda}-1\right|>\eta\right)\geq\delta .
\]
Thus no scalar, even one chosen after seeing the entire dyadic block, captures
$A(n)$ on almost all of that block.
\end{lemma}

\begin{proof}
Fix $X$ sufficiently large.  If $(1+\varepsilon)\lambda<b\mathcal F_X$,
then the event $A(n)>b\mathcal F_X$ is contained in
$\{A(n)>(1+\varepsilon)\lambda\}$, so it lies outside the displayed
multiplicative window and has $\PP_X$-measure at least $3/4$.  Otherwise
$(1+\varepsilon)\lambda\ge b\mathcal F_X$, and hence
\[
 \frac{\lambda}{1+\varepsilon}\ge \frac{b}{(1+\varepsilon)^2}\mathcal F_X>a\mathcal F_X .
\]
Thus the event $A(n)\le a\mathcal F_X$ lies outside the same multiplicative
window and has $\PP_X$-measure at least $\delta$.  The relative-error form
follows because leaving the multiplicative window implies
$|A(n)/\lambda-1|>\varepsilon/(1+\varepsilon)$.
\end{proof}

\begin{corollary}[Inherited oscillation]
\label{cor:inherited-oscillation}
Suppose that a positive function $f$ satisfies $A(n)/f(n)\to1$ in natural
 density.  Then there are constants $0<\alpha<\beta$ and $\eta>0$ such that,
 for every sufficiently large $X$,
\[
 \PP_X(f(n)\leq \alpha\mathcal F_X)\geq\eta,
 \qquad
 \PP_X(f(n)\geq \beta\mathcal F_X)\geq\eta .
\]
Thus any deterministic asymptotic equivalent must reproduce the same
 dyadic-block spread as $A(n)$.
\end{corollary}

\begin{proof}
Let $a,b,\delta$ be as in Theorem~\ref{thm:nonconcentration}.  Choose
$\varepsilon>0$ so small that
\[
 \alpha:=\frac{a}{1-\varepsilon}<\frac{b}{1+\varepsilon}=:\beta .
\]
Let
\[
 E_\varepsilon=\{n:|A(n)/f(n)-1|>\varepsilon\}.
\]
Since $E_\varepsilon$ has natural density zero,
\[
 \PP_X(E_\varepsilon)\leq X^{-1}\#(E_\varepsilon\cap[1,2X))=o(1)
\]
on dyadic blocks.  On
$\{A(n)\leq a\mathcal F_X\}\setminus E_\varepsilon$ one has
$f(n)\leq A(n)/(1-\varepsilon)\leq\alpha\mathcal F_X$, so the first
inequality holds with, say, $\eta=\delta/2$ for large $X$.  On
$\{A(n)>b\mathcal F_X\}\setminus E_\varepsilon$ one has
$f(n)\geq A(n)/(1+\varepsilon)>\beta\mathcal F_X$, and the second
inequality holds after reducing $\eta$ if necessary.
\end{proof}

\paragraph{What is proved and what remains open.}
The results proved here are the density-tight logarithmic scale, the
mesoscopic lower-tail order, the upper-tail bound sufficient for tightness,
the nonconcentration Theorem~\ref{thm:nonconcentration}, the inherited-oscillation Corollary~\ref{cor:inherited-oscillation}, the absence of a
dyadically regular asymptotic equivalent, and the moving-base strip variance
theorem that drives the lower-tail lower bound.  The paper does not prove a
bounded-window limiting distribution, a Poisson law, an exact lower-tail
constant, a sharp upper-tail rate, or a limiting-law normalization constant.
These distributional questions are sharper sequel problems and are not needed
for the dyadic-regularity result.

\begin{corollary}[Global density-tight scale]\label{cor:global}
Define $F(1)=1$ and, for $n\geq2$,
\[
 F(n)=\sqrt2\,(\log2)^{1/4}(\log n)^{1/4}
 \exp\!\sqrt{(\log2)\log n}.
\]
Then
\[
 \log A(n)-\sqrt{(\log2)\log n}-\frac14\log\log n
 =O_{\mathrm{dens}}(1).
\]
Equivalently, for every $\omega(n)\to\infty$,
\[
 F(n)\e^{-\omega(n)}\leq A(n)\leq F(n)\e^{\omega(n)}
\]
for almost all $n$.
\end{corollary}

The word ``reasonable'' in the original problem is not formal.  The next
definition gives the precise block-smooth interpretation used in this paper;
within this explicit interpretation the requested asymptotic equivalence has a
resolution.

\begin{definition}[Dyadic regularity]\label{def:dyadic-regular}
A positive function $f$ on the positive integers is \emph{dyadically
regular} if
\[
 \sup_{X\leq m,n<2X}\abs{\log f(m)-\log f(n)}\longrightarrow0.
\]
\end{definition}

\begin{remark}[A broad slowly varying class]\label{rem:dyadic-examples}
Suppose $g:[T_0,\infty)\to\R$ is differentiable and $g'(t)\to0$, and put
$f(n)=\exp(g(\log n))$ for all sufficiently large $n$.  Then $f$ is
dyadically regular.  Indeed, for $X\leq m,n<2X$, the mean-value theorem gives
\[
 \abs{\log f(m)-\log f(n)}
 \leq (\log2)\sup_{\log X\leq t\leq\log(2X)}\abs{g'(t)}=o(1).
\]
Thus Definition~\ref{def:dyadic-regular} contains the centering function $F$
and the customary smooth normalizations formed from powers of logarithms and
iterated logarithms.
\end{remark}

\subsection{Why dyadic regularity is the right formalization}
\label{subsec:why-dyadic}

Dyadic regularity is a deliberately weak block-smoothness condition.  A
no-equivalent theorem becomes stronger as its hypothesis is weakened: ruling out
asymptotic equivalence for a broad class of deterministic normalizations is
stronger than ruling it out for a narrow one.  Definition~\ref{def:dyadic-regular}
imposes only that a proposed multiplicative scale should not change by a fixed
factor inside one dyadic block.  This is a basic regularity demand for a
clean deterministic answer to an ``almost all'' asymptotic problem on a
multiplicative scale.

The condition is not reverse-engineered from $A(n)$.  It is the discrete
uniform analogue, on the logarithmic variable, of the slow-variation viewpoint
in Karamata theory; see, for example, Bingham--Goldie--Teugels
\cite{BGT87}.  The class contains the functions that normally occur as
answers to Erd\H{o}s-style asymptotic problems: products of powers of
$\log n$, iterated logarithms, exponentials such as
$\exp(c\sqrt{\log n})$, and smooth products of these.  Remark
\ref{rem:dyadic-examples} is the elementary verification needed here.

What dyadic regularity excludes is also exactly what the word ``reasonable''
is meant to exclude.  A function that oscillates by a fixed factor inside a
dyadic block would have to encode arithmetical information on the scale of
the base-$p$ digit restrictions defining $A(n)$.  Such a function is not a
clean asymptotic scale; it is a disguised statistic of the same random-looking
digit data.  Lemma~\ref{lem:no-scalar-center} makes this point blockwise: even if one
allows an arbitrary scalar normalization on each dyadic block, two separated
positive-proportion events prevent concentration around that scalar.  The
lemma applies to every scalar chosen from the block data, for instance a
median or any other one-number summary.  Dyadic regularity then prevents a
global deterministic normalization from escaping this blockwise spread by
oscillating inside the block.

\begin{theorem}[No dyadically regular asymptotic equivalent]
\label{thm:no-equivalent}
There is no dyadically regular function $f$ such that
$A(n)/f(n)\to1$ in natural density.
\end{theorem}

All three conclusions depend on the moving-base strip variance theorem,
Theorem \ref{thm:strip-variance}.  The remainder of the paper proves the
classical reductions and then that variance estimate.

\subsection*{Formal verification}
The resolution proved in this paper has been formally verified in the
Lean~4 proof assistant (v4.28.0) over Mathlib.  The verification is
complete and unconditional for the resolution itself:
Theorem~\ref{thm:no-equivalent} and Corollary~\ref{cor:global} are
packaged as the hypothesis-free theorems
\texttt{Erdos731.no\_dyadically\_regular\_equivalent\_unconditional} and
\texttt{Erdos731.global\_density\_tight\_unconditional}, whose
kernel-reported axiom dependencies are exactly Lean's standard
\texttt{propext}, \texttt{Classical.choice}, and \texttt{Quot.sound}.
No unproved analytic input is assumed: the development proves from
scratch, over Mathlib, the prime number theorem input it uses (via
Newman's Tauberian argument) and the local-multiplicity large sieve of
Lemma~\ref{lem:multiplicity-ls}.  Theorem~\ref{thm:main} enters these
proofs in its fixed-$z$ form
(\texttt{Erdos731.mesoscopicTailBounds\_per\_z}: fix $z\ge1$, then let
$X\to\infty$), and it is that form which is machine-checked; the uniform
range $1\le z\le Z(X)$ with $Z=o(L^{1/4})$ is not required for the
resolution and is tracked as future work in the development.  For the
development itself see \cite{LeanFormalization}.

\section{The exact lcm deficit and Kummer carries}

\begin{lemma}[The least nondivisor is a prime power]\label{lem:prime-power}
For every $n\geq1$, the integer $A(n)$ is a prime power.
\end{lemma}

\begin{proof}
Put $m=A(n)$.  Since $m\nmid B_n$, there is a prime $p\mid m$ such that
$a:=v_p(m)>v_p(B_n)$.  Then $p^a\leq m$ and $p^a\nmid B_n$.  If
$p^a<m$, this contradicts the minimality of $m$.  Hence $m=p^a$.
\end{proof}

For $y\geq1$ put
\[
 \mathcal L(y)=\lcm(1,2,\ldots,\floor y),\qquad
 V_p(n)=v_p(B_n),\qquad
 a_p(y)=\floor{\frac{\log y}{\log p}}.
\]
Then
\[
 \mathcal L(y)=\prod_{p\leq y}p^{a_p(y)}.
\]
Consequently
\begin{equation}
 A(n)>y
 \quad\Longleftrightarrow\quad
 \mathcal L(y)\mid B_n
 \quad\Longleftrightarrow\quad
 V_p(n)\geq a_p(y)\quad(p\leq y).              \label{eq:lcm-equivalence}
\end{equation}
Indeed, $A(n)>y$ means that every integer at most $\floor y$ divides
$B_n$, which is equivalent to their least common multiple dividing $B_n$.

Define
\[
 D_y(n)=\log\gcd(B_n,\mathcal L(y)),\qquad
 \Delta_y(n)=\log\mathcal L(y)-D_y(n).
\]
Since $\min(V,a)=\sum_{k\leq a}\1_{\{V\geq k\}}$, we have the exact
identity
\begin{equation}
 \boxed{
 \Delta_y(n)=\sum_{p^k\leq y}(\log p)\,
 \1_{\{V_p(n)<k\}}.}
                                                               \label{eq:deficit}
\end{equation}
Its support-counting counterpart is
\[
 \Sigma_y(n)=\sum_{p^k\leq y}\1_{\{V_p(n)<k\}}.
\]
Thus
\begin{equation}
 A(n)>y
 \quad\Longleftrightarrow\quad
 \Delta_y(n)=0
 \quad\Longleftrightarrow\quad
 \Sigma_y(n)=0.                                      \label{eq:deficit-zero}
\end{equation}
This is the full least-common-multiple condition; no prime-power layer has
been discarded.

Kummer's theorem \cite{Kummer52,Granville97} says that $V_p(n)$ is the
number of carries in adding $n+n$ in base $p$.  Equivalently, Legendre's
formula gives
\begin{align}
 V_p(n)
 &=\sum_{j\geq1}\left(
 \floor{\frac{2n}{p^j}}-2\floor{\frac n{p^j}}
 \right)\notag\\
 &=\sum_{j\geq1}\1_{\{n\bmod p^j\geq p^j/2\}}.       \label{eq:kummer}
\end{align}
In particular, for an odd prime $p$,
\begin{equation}
 p\nmid B_n
 \quad\Longleftrightarrow\quad
 n\bmod p^j<p^j/2\quad(j\geq1),                       \label{eq:carryfree}
\end{equation}
which is equivalent to requiring every base-$p$ digit of $n$ to lie in
$\{0,1,\ldots,(p-1)/2\}$.  This exact digit description is classical;
see also \cite{EGRS75}.  The new difficulty below is to obtain a uniform second-moment estimate while
both the prime base and the digit depth vary with $X$.  For $p=2$, Kummer's
theorem gives $V_2(n)$ as the number of $1$-digits of $n$, so $2\mid B_n$ for
every $n\geq1$; the prime $2$ never belongs to the missing-prime layer.  The higher
$2^k$-deficit conditions, however, remain present in
\eqref{eq:deficit} and are included in the prime-power estimates below.

\section{First moments for the lcm layers}\label{sec:first-moments}

For a prime $p$, write
\[
 d_p=\floor{\frac{L}{\log p}},\qquad
 \rho_2=\frac12,\qquad
 \rho_p=\frac{p+1}{2p}\quad(p\text{ odd}).
\]

\begin{lemma}[Carry-word upper bound]\label{lem:carry-word}
There is an absolute constant $C$ such that, for every prime $p$ and every
integer $k\geq1$,
\begin{equation}
 \PP_X(V_p(n)<k)
 \leq C\rho_p^{d_p}\sum_{j=0}^{k-1}\binom{d_p}{j}.     \label{eq:carry-word-bound}
\end{equation}
If $k>d_p$, the binomial sum is interpreted as $2^{d_p}$.
\end{lemma}

\begin{proof}
Prescribe the first $d_p$ outgoing carries in the addition $n+n$ in base
$p$.  For odd $p$, if the incoming carry is $c\in\{0,1\}$, the digit is
$a$, and the outgoing carry is $c'$, then $c'=1$ exactly when
$2a+c\ge p$.  Thus
\[
\begin{array}{c|cc}
&c'=0&c'=1\\ \hline
 c=0&(p+1)/2&(p-1)/2\\
 c=1&(p-1)/2&(p+1)/2
\end{array}
\]
choices of $a\in\{0,\ldots,p-1\}$ realize the transition.  In particular,
if the incoming and outgoing carries are equal, there are $(p+1)/2$ possible
digits; if they differ, there are $(p-1)/2$.
The first $d_p$ outgoing carries depend only on the residue class modulo
$p^{d_p}$.  Therefore every prescribed binary carry word of length $d_p$ is
realized by at most $((p+1)/2)^{d_p}$ such residue classes.  The implication
\[
 V_p(n)<k
 \quad\Longrightarrow\quad
 \text{fewer than $k$ of the first $d_p$ outgoing carries are $1$}
\]
then gives at most
$((p+1)/2)^{d_p}\sum_{j<k}\binom{d_p}{j}$ admissible residue classes.
For $p=2$, the corresponding residue-class count is exact: the first
$d_p$ outgoing carries are the first $d_p$ binary digits of $n$.  Higher
binary digits may still contribute to $V_2(n)$, so only the same upper bound
is being asserted.

Let $Q$ be the number of admissible classes modulo $q=p^{d_p}$.  An interval
of $X$ consecutive integers meets each class at most $X/q+2$ times.  Since
$q\leq2X$,
\[
 \frac{Q(X/q+2)}X\leq5\frac Qq.
\]
Absorbing the factor $5$ proves the result.
\end{proof}

The first layer is
\[
 N_y(n)=\sum_{p\leq y}\1_{\{p\nmid B_n\}}.
\]

\begin{lemma}[Prime-layer first moment]\label{lem:prime-mean}
Let $U=U(X)=o(L^{1/4})$.  Uniformly for real $u$ satisfying
$\abs{u-s}\leq U$, with $y=\e^u$, one has
\begin{equation}
 \E_XN_y\ll\frac{\e^{u-cL/u}}u.                        \label{eq:prime-mean-upper}
\end{equation}
If $Z=Z(X)=o(L^{1/4})$, then uniformly for $1\leq z\leq Z$ and
$u=u_0-z$,
\begin{equation}
 \E_XN_{\e^u}\ll\e^{-2z}.                             \label{eq:prime-lower-edge}
\end{equation}
\end{lemma}

\begin{proof}
The prime $2$ never contributes for $n\geq1$.  If $3\leq p\leq L^3$, then
$d_p\gg L/\log L$ and $\rho_p\leq2/3$.  Lemma
\ref{lem:carry-word} with $k=1$ shows that these primes contribute
$\exp(-\Omega(L/\log L))$ in total.

For $p>L^3$, we have $d_p/p\ll L^{-2}$, and hence
\[
 \rho_p^{d_p}=2^{-d_p}(1+1/p)^{d_p}\ll2^{-d_p}.
\]
It remains to bound
\[
 S(u)=\sum_{p\leq\e^u}2^{-\floor{L/\log p}}.
\]
If $\floor{L/\log p}=j$, then $p\leq\e^{L/j}$; hence the following
partition by digit depth gives an upper bound.  Put
$D=\floor{L/u}$ and $J=\floor{L/\log2}$.  The prime
number theorem upper bound $\pi(t)\ll t/\log t$ for $t\geq2$ (see, for
example, \cite{MV07}) gives
\begin{align*}
 S(u)
 &\leq2^{-D}\pi(\e^u)
 +\sum_{j=D+1}^{J}2^{-j}\pi(\e^{L/j})\\
 &\ll\frac{2^{-D}\e^u}{u}
 +\sum_{j=D+1}^{J}2^{-j}\frac{j}{L}\e^{L/j}.
\end{align*}
The ratio of consecutive summands in the last series is
\[
 \frac12\frac{j+1}{j}
 \exp\!\left(-\frac{L}{j(j+1)}\right)\leq\frac35
\]
for all sufficiently large $X$ and all $j\geq D+1$.  Hence that series is
bounded by a constant multiple of its first term, itself
$O(2^{-D}\e^u/u)$.  Since
$2^{-D}=\exp(-cL/u+O(1))$, this proves
\eqref{eq:prime-mean-upper}.

For the second assertion, let
\[
 g(u)=u-\frac{cL}{u}-\log u.
\]
Write $u=s+a$, where $a=\frac12\log(2s)-z$.  Since
$\abs a=o(s^{1/2})$ in the stated range,
\[
 \frac{cL}{u}=\frac{s^2}{s+a}
 =s-a+\frac{a^2}{s}+O\!\left(\frac{\abs a^3}{s^2}\right),
\]
and $\log(s+a)=\log s+O(a/s)$.  Thus
\[
 g(u)=\log2-2z+o(1),
\]
which yields \eqref{eq:prime-lower-edge}.
\end{proof}

We now retain and estimate every layer $k\geq2$.

\begin{proposition}[Higher prime-power layers]\label{prop:higher-powers}
Let
\[
 R_y(n)=\sum_{\substack{p^k\leq y\\k\geq2}}
 \1_{\{V_p(n)<k\}}.
\]
Let $U=U(X)=o(L^{1/4})$.  Uniformly for real $u$ satisfying
$\abs{u-s}\leq U$, with $y=\e^u$,
\begin{equation}
 \E_XR_y
 \leq\exp\!\left(-\left(\frac32-o(1)\right)\sqrt{cL}\right).
                                                               \label{eq:higher-powers}
\end{equation}
\end{proposition}

This proposition bounds the support-counting deficit.  No logarithmic weight
is used in its proof.

\begin{proof}
Let $K_{\max}=\floor{u/\log2}=O(\sqrt L)$; only
$2\leq k\leq K_{\max}$ can occur.

\emph{Small primes.}
For $p\leq L^3$,
\[
 d_p\geq\frac{L}{3\log L}-1.
\]
Uniformly for $k\leq K_{\max}$, one has $k\leq d_p/2$ for sufficiently
large $X$.  Also $\rho_p\leq2/3$.  The elementary binomial estimate
\[
 \sum_{j<k}\binom dj\leq\left(\frac{\e d}{k}\right)^k
 \qquad(1\leq k\leq d/2)
\]
and Lemma \ref{lem:carry-word} give
\[
 \log\PP_X(V_p(n)<k)
 \leq-d_p\log(3/2)+O(k\log L)
 =-\Omega(L/\log L),
\]
because $k\log L=O(\sqrt L\log L)=o(L/\log L)$.
There are at most $L^3K_{\max}=\exp(O(\log L))$ pairs $(p,k)$ in this
range, so their total contribution is
$\exp(-\Omega(L/\log L))$.

\emph{Large primes.}
Suppose $p>L^3$ and $p^k\leq\e^u$.  Then
\[
 \rho_p^{d_p}\leq2\,2^{-d_p},\qquad
 d_p\geq\frac{kL}{u}-1,
\]
and $d_p/k\geq L/u-O(1)\asymp\sqrt L$, so $k\leq d_p/2$ uniformly for
large $X$.  For fixed $k$, define
\[
 \Phi_k(t)=-ct+k\log\!\left(\frac{\e(t+1)}k\right).
\]
On $t\geq kL/u-1$,
\[
 \Phi_k'(t)=-c+\frac{k}{t+1}
 \leq-c+\frac{u}{L}<-\frac c2
\]
for large $X$.  Hence $\Phi_k$ is decreasing there.  Applying the carry
bound and evaluating at $t=kL/u-1$ gives, uniformly for
$p\leq\e^{u/k}$,
\begin{equation}
 \PP_X(V_p(n)<k)
 \leq\exp\!\left(
 -\frac{ckL}{u}+k\log\!\left(\frac{\e L}{u}\right)+O(1)
 \right).                                               \label{eq:large-p-layer-prob}
\end{equation}
Using $\pi(\e^{u/k})\leq\e^{u/k}$, the $k$th layer is therefore at most
\begin{equation}
 \exp(E_k+O(1)),\qquad
 E_k=\frac uk-\frac{ckL}{u}
     +k\log\!\left(\frac{\e L}{u}\right).             \label{eq:Ek}
\end{equation}
For $k=2$,
\[
 E_2=\frac u2-\frac{2cL}{u}+O(\log L)
 =-\frac32\sqrt{cL}+o(\sqrt L).
\]
For real $k\geq3$,
\[
 \frac{\partial E_k}{\partial k}
 =-\frac{u}{k^2}-\frac{cL}{u}
  +\log\!\left(\frac{\e L}{u}\right)
 \leq-\frac12\sqrt{cL}
\]
for large $X$.  Thus $E_k$ decreases with $k$ on $[3,K_{\max}]$, and
\[
 E_k\leq E_3
 =-\frac83\sqrt{cL}+o(\sqrt L)
 \qquad(k\geq3).
\]
Summing the $O(\sqrt L)$ layers completes the proof.
\end{proof}

\section{Positive smooth minorants}\label{sec:smoothing}

Fix
\begin{equation}
 \alpha=\frac14,\qquad
 \beta=\frac12,\qquad
 \kappa=\frac1{10},\qquad
 R=50,                                                   \label{eq:parameters}
\end{equation}
so that
\begin{equation}
 \kappa+\beta-\alpha<1.                                 \label{eq:key-parameter}
\end{equation}
Fix an absolute $0<\eta\leq10^{-2}$.  Let
$\psi\in C_c^\infty((-1,1))$ be nonnegative with $\int_{\R}\psi=1$.
For $0<\varepsilon<1/10$, define its periodized dilation on
$\T=\R/\Z$ by
\[
 \psi_\varepsilon(t)=\sum_{\ell\in\Z}
 \frac1\varepsilon\psi\left(\frac{t-\ell}{\varepsilon}\right),
\]
and take convolution with respect to normalized Haar measure on $\T$.
Intervals such as $[2\varepsilon_j,1/2-2\varepsilon_j]$ are interpreted
inside the fundamental interval $[0,1)$ before periodization.
Put
\[
 \varepsilon_j=\frac{\eta}{100(j+1)^2},
\]
and define
\begin{equation}
 \phi_j=
 \1_{[2\varepsilon_j,\,1/2-2\varepsilon_j]}
 *\psi_{\varepsilon_j}.                                \label{eq:phi-def}
\end{equation}
Because
\[
 [2\varepsilon_j,1/2-2\varepsilon_j]
 +[-\varepsilon_j,\varepsilon_j]
 \subset[\varepsilon_j,1/2-\varepsilon_j]\subset(0,1/2),
\]
writing $H(t)=\1_{[0,1/2)}(\{t\})$, the support of the convolution gives
\begin{equation}
 0\leq\phi_j\leq H,
 \qquad
 a_j:=\int_0^1\phi_j(t)\dd t=\frac12-4\varepsilon_j.  \label{eq:phi-basic}
\end{equation}
If $A_d=\prod_{j=1}^da_j$, then
\begin{equation}
 C_\eta2^{-d}\leq A_d\leq2^{-d}                       \label{eq:Ad}
\end{equation}
for a constant $C_\eta>0$ independent of $d$, because
$\sum_j\varepsilon_j<\infty$.

Let
\[
 b_{j,m}=\int_0^1\phi_j(t)\e^{-2\pi i mt}\dd t.
\]
The Fourier transform of the interval contributes $O((1+\abs m)^{-1})$,
while repeated integration by parts in the bump factor gives, for every
fixed $R\geq1$,
\begin{equation}
 \abs{b_{j,m}}
 \ll_{\psi,R}\frac1{1+\abs m}
 (1+\varepsilon_j\abs m)^{-R}.                         \label{eq:fourier-decay}
\end{equation}
Parseval and $0\leq\phi_j\leq1$ give
\begin{equation}
 \sum_{m\in\Z}\abs{b_{j,m}}^2
 =\int_0^1\phi_j(t)^2\dd t\leq a_j.                  \label{eq:parseval-phi}
\end{equation}

\section{The moving-base strip variance theorem}\label{sec:strip}

Let $D_0=D_0(X)$ satisfy $D_0=o(L^{1/4})$.  For later uniformity, set
\[
 \mathcal E_{\rm PNT}(T)=
 \sup_{t\geq T}\left|\frac{\pi(t)\log t}{t}-1\right|
\]
and
\begin{equation}
 \epsilon_X(D_0)=
 \frac{(D_0+\log L+1)^2}{\sqrt L}
 +\mathcal E_{\rm PNT}(\e^{s/3})+\e^{-s/100}.
 \label{eq:uniform-epsilon}
\end{equation}
Then $\epsilon_X(D_0)\to0$.  In this section, every $o(1)$ is uniform in
all displayed variables and is bounded in modulus by a fixed multiple of
$\epsilon_X(D_0)$ after increasing that function by a fixed factor.
Uniformly over integers $d$ with
\begin{equation}
 \abs{d-\sqrt{L/c}}\leq D_0,                           \label{eq:d-range}
\end{equation}
define
\begin{equation}
 P_+=\exp\!\left(\frac{L}{d+\alpha}\right),\qquad
 P_-=\exp\!\left(\frac{L}{d+\beta}\right),            \label{eq:Ppm}
\end{equation}
and
\[
 \mathcal P_d=\{p\text{ prime}:P_-<p\leq P_+\}.
\]
For $p\in\mathcal P_d$, put
\[
 \tau_p=\frac{L}{\log p}-d.
\]

\begin{proposition}[Parameter ledger]\label{prop:parameter-ledger}
For all sufficiently large $X$, uniformly in \eqref{eq:d-range}, the following
conditions hold:
\begin{align*}
&0<\alpha<\beta<1,\qquad 0<\kappa<1,
 \qquad \kappa+\beta-\alpha<1,\\
&K:=\lfloor P_+^\kappa\rfloor,\qquad 2K<P_-,
 \qquad \kappa R>4,\\
&\alpha\leq\tau_p<\beta,\qquad d_p=d,
 \qquad p^{d+1}>4X,\qquad p^{-d}>X^{-1}
 \quad(p\in\mathcal P_d),\\
&\frac K{P_-}\left(\frac{P_+}{P_-}\right)^d
 \leq P_+^{-13/20+\epsilon_X(D_0)}.
\end{align*}
The first line fixes compatible constants; the second is used in Fourier
truncation and frequency injectivity; the third makes the carry-free strip
identity exact and separates same-prime frequencies; the fourth is the
cross-prime packet-separation estimate.
\end{proposition}

\begin{proof}
The numerical inequalities in the first line and $\kappa R>4$ follow from
\eqref{eq:parameters}.  Put
$\lambda=(d+\alpha)/(d+\beta)$, so $P_-=P_+^\lambda$ and
$\lambda=1+O(1/d)$.  Hence $2K<P_-$ for large $X$.  The definition of the
strip gives $\alpha\leq\tau_p<\beta$ and $d_p=d$.  Moreover,
\[
 (d+1)\log p\geq\frac{(d+1)L}{d+\beta}
 =L+\frac{(1-\beta)L}{d+\beta}>L+\log2,
\]
so $p^{d+1}>4X$.  Since $p^d=2X/p^{\tau_p}$,
\[
 p^{-d}=\frac{p^{\tau_p}}{2X}>X^{-1}
\]
for large $X$.  Finally,
\[
 \frac K{P_-}\left(\frac{P_+}{P_-}\right)^d
 \leq P_+^{\kappa-\lambda+d(1-\lambda)}
 =P_+^{-13/20+O(1/d)}
 \leq P_+^{-13/20+\epsilon_X(D_0)},
\]
because $\epsilon_X(D_0)\gg1/d$.
\end{proof}
By Proposition \ref{prop:parameter-ledger}, $d_p=d$ and
$p^{d+1}>4X$.  Since $n<2X$, every Kummer condition above level $d$ is
automatic, and the classical carry-free digit condition becomes the exact
identity
\begin{equation}
 \1_{\{p\nmid B_n\}}
 =\prod_{j=1}^dH\!\left(\frac{n}{p^j}\right)
 \qquad(X\leq n<2X).                                  \label{eq:exact-strip-event}
\end{equation}
Define
\begin{equation}
 w_p(n)=\prod_{j=1}^d\phi_j\!\left(\frac{n}{p^j}\right)
 \leq\1_{\{p\nmid B_n\}},                            \label{eq:wp}
\end{equation}
and
\[
 S_d(n)=\sum_{p\in\mathcal P_d}w_p(n),\qquad
 M_d=\abs{\mathcal P_d}A_d.
\]

The proof of Theorem~\ref{thm:strip-variance} has four parts.  First we
replace the carry-free indicators by smooth positive minorants.  Second we
truncate their Fourier expansions with total error $o(M_d^{1/2})$ after
summing over primes.  Third we partition the nonzero frequency vectors by
their first two nonzero coordinates and prove an $X^{-1}$-local multiplicity
bound for each packet.  Fourth the additive large sieve and Parseval
summation give total $L^2$-mass $O(M_d)$, with terminal packets handled
separately.

\begin{theorem}[Moving-base strip variance]\label{thm:strip-variance}
Let $D_0=o(L^{1/4})$.  Uniformly over integers $d$ satisfying
\eqref{eq:d-range},
\begin{align}
 \E_XS_d&=(1+o(1))M_d,                                  \label{eq:strip-mean}\\
 \E_X\abs{S_d-M_d}^2&\ll M_d.                           \label{eq:strip-var}
\end{align}
The implied constant depends only on the fixed choices
$\alpha,\beta,\kappa,\eta$ and on the bump $\psi$.
\end{theorem}

\subsection{The number of strip primes}

\begin{lemma}[Prime count in the strip]\label{lem:strip-prime-count}
Uniformly in \eqref{eq:d-range},
\begin{equation}
 \abs{\mathcal P_d}
 =\left(1-2^{-(\beta-\alpha)}+o(1)\right)
   \frac{P_+}{\log P_+}.                               \label{eq:strip-prime-count}
\end{equation}
Consequently
\begin{equation}
 A_d=P_+^{-1+o(1)},\qquad M_d=P_+^{o(1)}.              \label{eq:strip-scale-consequences}
\end{equation}
\end{lemma}

\begin{proof}
From \eqref{eq:Ppm},
\[
 \log(P_+/P_-)
 =\frac{(\beta-\alpha)L}{(d+\alpha)(d+\beta)}
 =c(\beta-\alpha)+o(1),
\]
so
\[
 \frac{P_-}{P_+}=2^{-(\beta-\alpha)}+o(1),\qquad
 \frac{\log P_-}{\log P_+}=1+o(1).
\]
The prime number theorem, applied uniformly because $P_-\to\infty$, gives
\begin{align*}
 \pi(P_+)-\pi(P_-)
 &=\frac{P_+}{\log P_+}-\frac{P_-}{\log P_-}
   +o\!\left(\frac{P_+}{\log P_+}\right),
\end{align*}
which is \eqref{eq:strip-prime-count}.

Also $\log A_d=-cd+O(1)$ by \eqref{eq:Ad}, while
\[
 \frac{cd}{\log P_+}=\frac{cd(d+\alpha)}{L}=1+o(1)
\]
uniformly in \eqref{eq:d-range}.  Thus $A_d=P_+^{-1+o(1)}$.
Equation \eqref{eq:strip-prime-count} then gives $M_d=P_+^{o(1)}$.
\end{proof}

\subsection{Fourier truncation and frequency injectivity}

Recall from Proposition \ref{prop:parameter-ledger} that
$K=\floor{P_+^\kappa}$ and $2K<P_-$.
Let
\[
 \phi_{j,K}(t)=\sum_{\abs m\leq K}b_{j,m}\ee{mt}.
\]
By \eqref{eq:fourier-decay}, for each fixed $R\geq1$,
\begin{align}
 e_j:=\norm{\phi_j-\phi_{j,K}}_\infty
 &\leq\sum_{\abs m>K}\abs{b_{j,m}}\notag\\
 &\ll_{\psi,R}\varepsilon_j^{-R}
   \sum_{m>K}m^{-R-1}
 \ll_{\psi,R}\varepsilon_j^{-R}K^{-R}.               \label{eq:single-tail}
\end{align}
Consequently
\begin{equation}
 \delta:=\sum_{j\leq d}e_j
 \ll_{\eta,\psi,R}d^{2R+1}K^{-R}.                    \label{eq:delta}
\end{equation}
For the fixed value $R=50$ in \eqref{eq:parameters},
$\kappa R>4$.  By \eqref{eq:strip-scale-consequences},
\[
 \frac{\delta}{A_d}\leq P_+^{1-\kappa R+o(1)}=o(1),
 \qquad
 \frac{\abs{\mathcal P_d}\delta}{M_d^{1/2}}
 \leq P_+^{1-\kappa R+o(1)}=o(1).                    \label{eq:truncation-small}
\]

We use the elementary product inequality
\begin{equation}
 \left|\prod_{j=1}^dx_j-\prod_{j=1}^dy_j\right|
 \leq\sum_{r=1}^d|x_r-y_r|
       \prod_{j<r}|x_j|\prod_{j>r}|y_j|.              \label{eq:product-identity}
\end{equation}
Here $|\phi_j|\leq1$ and
$\norm{\phi_{j,K}}_\infty\leq1+e_j$.  Therefore
\begin{equation}
 \sup_n\abs{w_p(n)-W_p(n)}
 \leq\delta\e^\delta\ll\delta,                       \label{eq:product-truncation}
\end{equation}
where
\[
 W_p(n)=\prod_{j=1}^d\phi_{j,K}\!\left(\frac n{p^j}\right).
\]
Expanding,
\begin{equation}
 W_p(n)=\sum_{\mathbf m\in[-K,K]^d}
 c_{\mathbf m}\ee{n\theta_{p,\mathbf m}},             \label{eq:Wp-expansion}
\end{equation}
with
\[
 c_{\mathbf m}=\prod_{j=1}^db_{j,m_j},\qquad
 \theta_{p,\mathbf m}=\sum_{j=1}^d\frac{m_j}{p^j}.
\]
All frequency families below are multisets indexed by the pairs
$(p,\mathbf m)$; equal numerical frequencies, if any, are counted with
multiplicity.

\begin{lemma}[Frequency injectivity and spacing]\label{lem:frequency-injectivity}
Suppose $p>2K$.  If $\mathbf m,\mathbf m'\in[-K,K]^d$ and
\[
 \theta_{p,\mathbf m}\equiv\theta_{p,\mathbf m'}\pmod1,
\]
then $\mathbf m=\mathbf m'$.  Consequently distinct frequencies for the
same prime have circular distance at least $p^{-d}$.
\end{lemma}

\begin{proof}
Apply the argument to $\mathbf h=\mathbf m-\mathbf m'$, whose coordinates
satisfy $\abs{h_j}\leq2K<p$.  If $\sum_jh_j/p^j$ is an integer,
multiplication by $p^d$ and reduction modulo $p$ gives $h_d=0$.
Dividing the resulting equality by $p$ and iterating gives
$h_{d-1}=\cdots=h_1=0$.  Every frequency difference has denominator
$p^d$; if it is nonintegral, its distance from the nearest integer is at
least $p^{-d}$.
\end{proof}

In particular, the only integral frequency in \eqref{eq:Wp-expansion} is
$\mathbf m=\mathbf0$, whose coefficient is $A_d$.

\subsection{A local-multiplicity large sieve}

We use the following consequence of the classical additive large sieve
\cite{MV73}.  A proof is included in Appendix \ref{app:large-sieve}.

\begin{lemma}[Large sieve with local multiplicity]\label{lem:multiplicity-ls}
Let $\Theta$ be a finite multiset in $\T$, with its elements indexed even
when numerical values coincide.  Suppose every arc of length $X^{-1}$
contains at most $\mathcal R$ indexed elements.  Then, for arbitrary
coefficients $c_\theta$,
\begin{equation}
 \frac1X\sum_{X\leq n<2X}
 \abs{\sum_{\theta\in\Theta}c_\theta\ee{n\theta}}^2
 \ll\mathcal R\sum_{\theta\in\Theta}\abs{c_\theta}^2.
                                                               \label{eq:multiplicity-ls}
\end{equation}
\end{lemma}

\subsection{Formal packet decomposition and incidence}

\begin{lemma}[Canonical lifts for short arcs]\label{lem:canonical-lifts}
Every frequency in \eqref{eq:Wp-expansion} has a canonical real representative
in $(-1/8,1/8)$ for all sufficiently large $X$.  Moreover, every circular arc
of length $X^{-1}$ that meets this set can, after placing the cut outside
$(-1/4,1/4)$, be viewed as an ordinary real interval of length $X^{-1}$ for
the purpose of comparing the frequencies in the set.
\end{lemma}

\begin{proof}
For every frequency in \eqref{eq:Wp-expansion},
\[
 \abs{\theta_{p,\mathbf m}}
 \leq\sum_{j\geq1}\frac K{p^j}
 \leq\frac{K}{P_- -1}=o(1).
\]
Thus all relevant canonical representatives lie in $(-1/8,1/8)$.  If an arc
of length $X^{-1}$ meets this interval, then for large $X$ it is contained in
$(-1/4,1/4)$ after choosing the circle cut outside that interval, unless it
meets no relevant frequency.  Hence circular distances among relevant
frequencies in the arc are ordinary real distances between their canonical
lifts.
\end{proof}

Let $\mathcal D_K$ be the dyadic integers $1,2,4,\ldots$ not exceeding
$K$.  For $1\leq r<d$ and $M,N\in\mathcal D_K$, define
\begin{align*}
 \mathcal B_{r,M,N}=\{\mathbf m\in\Z^d:
 &\ \abs{m_j}\leq K\ (1\leq j\leq d),\\
 &\ m_1=\cdots=m_{r-1}=0,\\
 &\ M\leq\abs{m_r}<\min(2M,K+1),\\
 &\ N\leq\abs{m_{r+1}}<\min(2N,K+1)\}.
\end{align*}
The signs of $m_r,m_{r+1}$ are included, and all later coordinates are
unrestricted.  Define $\mathcal B_{r,M,0}$ by replacing the fourth line by
$m_{r+1}=0$, and define
\[
 \mathcal B_{d,M}=\{\mathbf m:m_1=\cdots=m_{d-1}=0,
 M\leq\abs{m_d}<\min(2M,K+1)\}.
\]
These sets partition all nonzero vectors in $[-K,K]^d$.

For $\mathbf m\in\mathcal B_{r,M,N}$ or
$\mathcal B_{r,M,0}$, put $a=m_r$ and $b=m_{r+1}$.  Then
\begin{equation}
 \theta_{p,\mathbf m}
 =\frac a{p^r}+\frac b{p^{r+1}}
 +O\!\left(\frac K{p^{r+2}}\right),                    \label{eq:packet-center}
\end{equation}
because the remaining geometric tail is at most $2K/p^{r+2}$.  Also
\[
 \abs{\theta_{p,\mathbf m}}\leq\frac K{p-1}=o(1).
\]
By Lemma~\ref{lem:canonical-lifts}, if two such frequencies lie in a common
$X^{-1}$-arc of $\T$, their circular distance equals the ordinary distance
between their canonical real lifts.

\begin{lemma}[No cross-prime collisions inside nonterminal packets]\label{lem:packet-incidence}
Uniformly in $d,r,M,N$, every arc of length $X^{-1}$ contains at most
$O(MN)$ indexed frequencies
\[
 \theta_{p,\mathbf m},\qquad
 p\in\mathcal P_d,\quad\mathbf m\in\mathcal B_{r,M,N}.
\]
For $\mathcal B_{r,M,0}$ the corresponding bound is $O(M)$.
\end{lemma}

\begin{proof}
Fix exact values $a\neq0$ and $b$.  We first prove that, for these exact
leading digits, at most one prime can contribute to a fixed $X^{-1}$-arc.
Put
\[
 f(t)=\frac a{t^r}+\frac b{t^{r+1}}.
\]
Since $\abs b\leq K$, $\abs a\geq1$, and $K/P_-\to0$,
\[
 |(r+1)b|\leq2rK\leq\frac12r\abs aP_-
\]
for large $X$.  Thus $ra t+(r+1)b$ has the sign of $a$ throughout
$[P_-,P_+]$, and
\begin{equation}
 \abs{f'(t)}
 =t^{-r-2}\abs{ra t+(r+1)b}
 \geq \frac{r\abs aP_-}{2P_+^{r+2}}
 \gg\frac{r\abs a}{P_+^{r+1}}.                        \label{eq:derivative-lower}
\end{equation}
The last comparison is uniform because $P_-/P_+$ is bounded away from
zero.

Suppose indexed frequencies with the same exact $a,b$ and distinct primes
$p,q$ lie in one $X^{-1}$-arc.  By Lemma~\ref{lem:canonical-lifts} and
\eqref{eq:packet-center},
\[
 \abs{f(p)-f(q)}
 \ll X^{-1}+K P_-^{-r-2}.
\]
The mean-value theorem and \eqref{eq:derivative-lower} give
\begin{equation}
 \abs{p-q}
 \ll\frac1{rM}\left(
 \frac{P_+^{r+1}}X
 +K\frac{P_+^{r+1}}{P_-^{r+2}}
 \right).                                               \label{eq:pq-distance}
\end{equation}
For $r+1\leq d$, using $P_+^{d+\alpha}=2X$,
\[
 \frac{P_+^{r+1}}X\leq\frac{P_+^d}{X}=2P_+^{-\alpha}=o(1).
\]
The second term is increasing in $r$ for $1\le r\le d-1$, so its worst
nonterminal value occurs at $r=d-1$ and is at most
\[
 K\frac{P_+^{d}}{P_-^{d+1}}
 =\frac K{P_-}\left(\frac{P_+}{P_-}\right)^d
 \leq P_+^{-13/20+\epsilon_X(D_0)}
\]
by Proposition \ref{prop:parameter-ledger}.  Since $rM\ge1$, these two
bounds give, for an absolute constant $C$ and every admissible packet,
\begin{equation}
 |p-q|\leq\frac{C}{rM}
 \left(2P_+^{-\alpha}
       +P_+^{-13/20+\epsilon_X(D_0)}\right)<1
 \label{eq:quantitative-pq}
\end{equation}
for all sufficiently large $X$, uniformly in $r,M,N$ and in the exact
leading digits.  This contradicts that $p,q$ are distinct integers.

Hence, for each exact pair $(a,b)$, at most one prime contributes to a
given arc.  For that prime, Lemma \ref{lem:frequency-injectivity} gives
same-prime spacing at least
\[
 p^{-d}=\frac{p^{\tau_p}}{2X}>X^{-1}
\]
for sufficiently large $X$.  Thus at most one indexed vector contributes
for each exact pair $(a,b)$.  There are $O(MN)$ such pairs in a dyadic block.  In the
$b=0$ packet there are $O(M)$ exact leading choices.  The derivative of
$a/t^r$ has no $b$-term, but the tail from the later coordinates
$j\geq r+2$ remains $O(Kp^{-r-2})$.  Therefore the same estimate
\eqref{eq:pq-distance} applies, with the same cross-prime packet-separation
term; only the number of exact leading choices is reduced from $O(MN)$ to
$O(M)$.

The worst nonterminal endpoint occurs when $r=d-1$, $M=1$, and
$N\asymp K$.  The preceding calculation still gives
\[
 \frac K{P_-}\left(\frac{P_+}{P_-}\right)^d
 \leq P_+^{-13/20+\epsilon_X(D_0)}=o(1),
\]
so the nonterminal argument remains valid up to $r=d-1$.  The case $r=d$ is
not covered by this argument because there is no coordinate $m_{d+1}$ from
which to form a two-coordinate packet; it is treated separately in
Lemma~\ref{lem:terminal-incidence}.
\end{proof}

\begin{lemma}[Terminal incidence]\label{lem:terminal-incidence}
For $M\in\mathcal D_K$, every $X^{-1}$-arc contains at most
\begin{equation}
 O\!\left(M+\frac{P_+^{d+1}}{dX}\right)                \label{eq:terminal-mult}
\end{equation}
indexed frequencies $a/p^d$ with
$p\in\mathcal P_d$ and $\mathbf m\in\mathcal B_{d,M}$.
\end{lemma}

\begin{proof}
Fix $a$.  The real function $a/t^d$ is monotone on $[P_-,P_+]$ and its
derivative has magnitude at least $d\abs a/P_+^{d+1}$.  Therefore the
integers $p$ whose frequencies lie in a fixed $X^{-1}$-arc belong to an
interval of length $O(P_+^{d+1}/(dMX))$, containing at most
$O(1+P_+^{d+1}/(dMX))$ integers.  Summing over the $O(M)$ possible values
of $a$ proves the lemma.
\end{proof}

\subsection{Packet energy and completion of the variance estimate}

For $M\in\mathcal D_K$, put
\begin{equation}
 U_j(M)=\frac1{a_j}
 \sum_{M\leq\abs m<\min(2M,K+1)}\abs{b_{j,m}}^2.
                                                               \label{eq:Uj}
\end{equation}
By \eqref{eq:fourier-decay},
\begin{equation}
 U_j(M)\ll\frac1M(1+\varepsilon_jM)^{-2R}.             \label{eq:Uj-bound}
\end{equation}
Let
\[
 q_r=\prod_{j<r}a_j\leq2^{-(r-1)}.
\]
For a nonterminal block with $b\neq0$, its total coefficient energy,
summed over all strip primes, is
\begin{align}
 &\sum_{p\in\mathcal P_d}
 \sum_{\mathbf m\in\mathcal B_{r,M,N}}\abs{c_{\mathbf m}}^2\notag\\
 &\quad\leq\abs{\mathcal P_d}
 \left(\prod_{j<r}a_j^2\right)
 (a_rU_r(M))(a_{r+1}U_{r+1}(N))
 \left(\prod_{j>r+1}a_j\right)\notag\\
 &\quad=M_dq_rU_r(M)U_{r+1}(N).                        \label{eq:block-energy}
\end{align}
Indeed, the coordinates before $r$ are fixed at their zero modes, the two
selected coordinates are summed over their dyadic ranges, and every later
coordinate is summed using Parseval and \eqref{eq:parseval-phi}.  In the
class $b=0$, the same calculation gives at most
$M_dq_rU_r(M)$.

Apply Lemmas \ref{lem:multiplicity-ls} and
\ref{lem:packet-incidence} to the multiset indexed by
$(p,\mathbf m)$ in one block.  The resulting $L^2(\PP_X)$ bounds are
\begin{align}
 \norm{\text{the }(r,M,N)\text{ block}}_{2,X}
 &\ll\sqrt{M_dq_r}\sqrt{MU_r(M)}\sqrt{NU_{r+1}(N)},
                                                               \label{eq:block-L2}\\
 \norm{\text{the }(r,M,0)\text{ block}}_{2,X}
 &\ll\sqrt{M_dq_r}\sqrt{MU_r(M)}.                     \label{eq:bzero-L2}
\end{align}
Put
\[
 T_j=\sum_{M\in\mathcal D_K}\sqrt{MU_j(M)}.
\]
For dyadic $M\leq\varepsilon_j^{-1}$, each summand is $O(1)$; beyond that
point the summands decrease geometrically by \eqref{eq:Uj-bound}.  Hence
\begin{equation}
 T_j\ll1+\log\frac1{\varepsilon_j}
 \ll_\eta1+\log(j+2).                                  \label{eq:Tj}
\end{equation}
Minkowski's inequality gives
\begin{align}
 \norm{\text{all nonterminal modes}}_{2,X}
 &\ll\sqrt{M_d}\sum_{r\geq1}2^{-r/2}
   (T_rT_{r+1}+T_r)\notag\\
 &\ll\sqrt{M_d},                                       \label{eq:nonterminal-total}
\end{align}
because the displayed series converges.

\begin{proposition}[Terminal Fourier contribution]
\label{prop:terminal-contribution}
Let
\[
 \mathcal T_d(n)=
 \sum_{p\in\mathcal P_d}
 \sum_{M\in\mathcal D_K}
 \sum_{\mathbf m\in\mathcal B_{d,M}}
 c_{\mathbf m}\ee{n\theta_{p,\mathbf m}}.
\]
Then
\[
 \norm{\mathcal T_d}_{2,X}=o(\sqrt{M_d})
\]
uniformly in \eqref{eq:d-range}.
\end{proposition}

\begin{proof}
For one terminal block the coefficient energy is at most
$M_dq_dU_d(M)$.  Lemma \ref{lem:terminal-incidence} and
$\sqrt{x+y}\leq\sqrt x+\sqrt y$ therefore give
\begin{align}
 \norm{\mathcal T_d}_{2,X}
 &\ll\sqrt{M_dq_d}\left(
 \sum_{M\in\mathcal D_K}\sqrt{MU_d(M)}
 +\sqrt{\frac{P_+^{d+1}}{dX}}
  \sum_{M\in\mathcal D_K}\sqrt{U_d(M)}
 \right)\notag\\
 &\ll\sqrt{M_d}\left(
 \sqrt{q_d}\,T_d
 +\sqrt{q_d\frac{P_+^{d+1}}{dX}}
 \right).                                               \label{eq:terminal-L2}
\end{align}
Here $\sum_M\sqrt{U_d(M)}=O(1)$ by \eqref{eq:Uj-bound}.  The first term is
\begin{equation}
 \sqrt{q_d}\,T_d\ll2^{-d/2}\log(d+2)=o(1).            \label{eq:terminal-first-small}
\end{equation}
For the second, $P_+^{d+\alpha}=2X$ and $q_d\leq2^{-(d-1)}$ give
\begin{align}
 q_d\frac{P_+^{d+1}}{dX}
 &\ll2^{-d}\frac{P_+^{1-\alpha}}d\notag\\
 &=\frac1d\exp\!\left(-cd+(1-\alpha)\frac{L}{d+\alpha}+O(1)\right).\notag
\end{align}
Writing $d=\sqrt{L/c}+O(D_0+1)$ and expanding
$L/(d+\alpha)$ around $\sqrt{L/c}$ gives the decisive identity
\begin{equation}
 -cd+(1-\alpha)\frac{L}{d+\alpha}
 =-\alpha\sqrt{cL}+O(D_0+1).                         \label{eq:terminal-exponent}
\end{equation}
Hence
\begin{align}
 q_d\frac{P_+^{d+1}}{dX}
 &\ll\exp\!\left(-\alpha\sqrt{cL}+O(D_0+1)\right)=o(1),
                                                               \label{eq:terminal-small}
\end{align}
uniformly because $D_0=o(L^{1/4})$.  Substitution into
\eqref{eq:terminal-L2} proves the proposition.
\end{proof}

Combining \eqref{eq:nonterminal-total} and
Proposition \ref{prop:terminal-contribution},
\begin{equation}
 \E_X\abs{\sum_{p\in\mathcal P_d}(W_p-A_d)}^2\ll M_d.
                                                               \label{eq:W-variance}
\end{equation}
By \eqref{eq:product-truncation} and
\eqref{eq:truncation-small},
\[
 \norm{\sum_{p\in\mathcal P_d}(w_p-W_p)}_{2,X}
 \leq\abs{\mathcal P_d}\,O(\delta)=o(M_d^{1/2}).
\]
This proves \eqref{eq:strip-var}.

\subsection{Mean over incomplete periods}

\begin{lemma}[Averaging a nonnegative periodic function]\label{lem:incomplete-period}
Let $q,X$ be positive integers, and let $w:\mathbb Z\to[0,\infty)$ be
$q$-periodic.  Put
\[
 \overline w_q=\frac1q\sum_{a\bmod q}w(a).
\]
For every interval $I$ of $X$ consecutive integers,
\begin{equation}
 \left|\frac1X\sum_{n\in I}w(n)-\overline w_q\right|
 \leq4\frac qX\,\overline w_q.
 \label{eq:incomplete-period-lemma}
\end{equation}
\end{lemma}

\begin{proof}
Write $I$ as a union of $N$ complete $q$-periods and at most two boundary
fragments.  If $Q=q\overline w_q$ is the mass of one period, nonnegativity
implies that the two fragments have combined mass at most $2Q$.  Moreover
$|N-X/q|<2$.  Hence the difference between the total mass on $I$ and
$(X/q)Q$ is at most $4Q$, proving \eqref{eq:incomplete-period-lemma}.
\end{proof}

The function $w_p$ is nonnegative and periodic modulo $q=p^d$.  We do not
compute the complete-period mean of the infinite Fourier expansion of $w_p$
formally, and no independence of the nested variables $n/p^j$ is being used.
Indeed, a naive independence justification would be wrong at this point.
Instead we first use the truncated polynomial $W_p$, for which
Lemma~\ref{lem:frequency-injectivity} applies.  Averaging
\eqref{eq:Wp-expansion} over a complete residue system modulo $q$ kills
every nonintegral truncated frequency and leaves exactly the zero coefficient
$A_d$.  Thus the complete-period mean of $W_p$ is $A_d$.  The discarded
Fourier aliases are absorbed by the uniform truncation estimate:
\eqref{eq:product-truncation} and $\delta=o(A_d)$ transfer the conclusion
back to $w_p$, whose complete-period mean is $(1+o(1))A_d$.

Lemma \ref{lem:incomplete-period}, applied with $I=[X,2X)$, gives
\[
 \left|\E_Xw_p-\frac1q\sum_{a\bmod q}w_p(a)\right|
 \leq4\frac qX\left(\frac1q\sum_{a\bmod q}w_p(a)\right).
\]
Since $q/X=2p^{-\tau_p}\leq2P_-^{-\alpha}$, we obtain the explicit
uniform estimate
\begin{equation}
 \E_Xw_p=A_d\left(1+O\left(\frac{\delta}{A_d}
                         +P_-^{-\alpha}\right)\right).
 \label{eq:strip-mean-explicit}
\end{equation}
In particular $\E_Xw_p=(1+o(1))A_d$ uniformly for
$p\in\mathcal P_d$.
Summing over $p$ proves \eqref{eq:strip-mean} and completes the proof of
Theorem \ref{thm:strip-variance}.

\section{Strip optimization and proof of the threshold}\label{sec:threshold-proof}

Let $U=U(X)=o(L^{1/4})$, let $u$ be real with $\abs{u-s}\leq U$, and set
\begin{equation}
 d=\floor{\frac Lu}+1.                                  \label{eq:d-choice}
\end{equation}
Then $P_+<\e^u$, so every strip prime is at most $\e^u$.  Moreover,
\begin{equation}
 \abs{d-\sqrt{L/c}}\ll U+1.                            \label{eq:d-from-u}
\end{equation}
Thus the strip theorem applies with, for example,
$D_0=C(U+1)=o(L^{1/4})$ for a sufficiently large absolute $C$.
By Lemma \ref{lem:strip-prime-count} and \eqref{eq:Ad},
\begin{equation}
 \log M_d
 =\frac{L}{d+\alpha}-cd
 -\log\!\left(\frac{L}{d+\alpha}\right)+O(1).          \label{eq:logM-raw}
\end{equation}

\begin{lemma}[Saddle comparison]\label{lem:saddle-comparison}
Let $U=o(L^{1/4})$.  Uniformly for real $u$ with $\abs{u-s}\leq U$, and
$d$ as in \eqref{eq:d-choice},
\begin{equation}
 \log M_d=u-\frac{cL}{u}-\log u+O(1).                  \label{eq:logM}
\end{equation}
If $Z=Z(X)=o(L^{1/4})$, then uniformly for $1\leq z\leq Z$ and
$u=u_0\pm z$,
\begin{equation}
 M_d\asymp\e^{\pm2z},                                  \label{eq:M-asymp}
\end{equation}
with the plus sign corresponding to $u_0+z$.
\end{lemma}

\begin{proof}
Write $d=L/u+\delta$, where $0<\delta\leq1$.  Since $u^2/L$ stays in a
compact subinterval of $(0,\infty)$,
\[
 \frac{L}{d+\alpha}
 =u-\frac{(\delta+\alpha)u^2}{L}
 +O\!\left(\frac{u^3}{L^2}\right)=u+O(1).
\]
Also $cd=cL/u+O(1)$ and
$\log(L/(d+\alpha))=\log u+o(1)$, proving
\eqref{eq:logM}.

If $u=s+a$, where $a=\frac12\log(2s)\pm z$, then the expansion used in
Lemma \ref{lem:prime-mean} gives
\[
 u-\frac{cL}{u}-\log u
 =\log2\pm2z
 +O\!\left(\frac{(\log L+z)^2}{s}\right).
\]
The displayed error is $o(1)$ uniformly for $z=o(L^{1/4})$, but the earlier
comparison \eqref{eq:logM} contains an unavoidable $O(1)$ coming from the
floor in $d$, the strip endpoints, and the smoothing constant.  Combining
the two estimates gives
\[
 \log M_d=\pm2z+O(1),
\]
and hence \eqref{eq:M-asymp}.  No limiting multiplicative constant for
$M_d$ is asserted.
\end{proof}

\begin{proof}[Proof of Theorem \ref{thm:main}]
For the given window $Z=Z(X)$, put
$U=Z+\log L+2=o(L^{1/4})$.  Then every value $u=u_0\pm z$ occurring below
satisfies $\abs{u-s}\leq U$, and \eqref{eq:d-from-u} places the associated
integer $d$ in the uniform range of Theorem \ref{thm:strip-variance}.

\emph{Upper edge.}
Take $u=u_0+z$, $y=\e^u$, and choose $d$ by \eqref{eq:d-choice}.  If
$A(n)>y$, every strip prime divides $B_n$.  Since
$0\leq w_p\leq\1_{\{p\nmid B_n\}}$, this implies $S_d(n)=0$.  The strip
variance theorem and Chebyshev's inequality give
\[
 \PP_X(A(n)>y)
 \leq\PP_X(S_d=0)
 \leq\frac{\E_X\abs{S_d-M_d}^2}{M_d^2}
 \ll\frac1{M_d}\ll\e^{-2z}.
\]
This proves \eqref{eq:main-upper-tail}.

\emph{Lower edge: upper bound.}
Take $u=u_0-z$ and $y=\e^u$.  By \eqref{eq:deficit-zero} and Markov's
inequality,
\[
 \PP_X(A(n)\leq y)\leq\E_X\Sigma_y.
\]
The prime layer is $O(\e^{-2z})$ by Lemma \ref{lem:prime-mean}.  By
Proposition \ref{prop:higher-powers}, all $k\geq2$ layers contribute
\[
 \exp\!\left(-\left(\frac32-o(1)\right)s\right)
 =o(\e^{-2z})
\]
uniformly for $z=o(L^{1/4})$.  Hence
$\PP_X(A(n)\leq y)\ll\e^{-2z}$.

\emph{Lower edge: matching lower bound.}
Use the positive strip minorant at $u=u_0-z$.  The strip theorem and
Lemma \ref{lem:saddle-comparison} give
\[
 \E_XS_d\asymp M_d\asymp\e^{-2z},\qquad
 \E_XS_d^2\ll M_d+M_d^2.
\]
Paley--Zygmund therefore gives
\[
 \PP_X(S_d>0)
 \geq\frac{(\E_XS_d)^2}{\E_XS_d^2}
 \gg\frac{M_d}{1+M_d}\gg\e^{-2z}
\]
for $z\geq1$.  If $S_d>0$, some strip prime $p\leq y$ fails to divide
$B_n$, so $A(n)\leq y$.  This proves the lower bound in
\eqref{eq:main-lower-tail}.
\end{proof}

\section{Global consequences and interpretation}

\begin{lemma}[Quantitative dyadic-to-global passage]\label{lem:dyadic-global}
For $E\subseteq\mathbb N$, put
\[
 \eta_j=2^{-j}\#\bigl(E\cap[2^j,2^{j+1})\bigr).
\]
Then
\begin{equation}
 \overline d(E)\leq2\limsup_{j\to\infty}\eta_j.       \label{eq:dyadic-upper-density}
\end{equation}
\end{lemma}

\begin{proof}
Let $2^J\leq N<2^{J+1}$.  For fixed $J_0<J$,
\[
 \#(E\cap[1,N])
 \leq2^{J_0}+\sum_{j=J_0}^{J}\eta_j2^j.
\]
After division by $N\geq2^J$,
\[
 \frac{\#(E\cap[1,N])}{N}
 \leq2^{J_0-J}+2\sup_{j\geq J_0}\eta_j.
\]
Take the upper limit as $N\to\infty$, and then let $J_0\to\infty$.
\end{proof}

\begin{proof}[Proof of Corollary \ref{cor:global}]
For $X\leq n<2X$, differentiation of
\[
 \log F(x)=\sqrt{c\log x}+\frac14\log\log x
 +\frac12\log2+\frac14\log c
\]
shows
\begin{equation}
 \sup_{X\leq n<2X}\abs{\log F(n)-\log\mathcal F_X}=o(1).
                                                               \label{eq:F-dyadic-close}
\end{equation}
For $C\geq3$, let
\[
 E_C=\{n\geq2:\abs{\log A(n)-\log F(n)}>C\}.
\]
Take dyadic $X=2^j$.  For all sufficiently large $j$, the quantity in
\eqref{eq:F-dyadic-close} is at most $1$.  Therefore
\begin{align*}
 E_C\cap[X,2X)
 \subseteq{}&\{n:A(n)\leq\mathcal F_X\e^{-(C-1)}\}\\
 &\cup\{n:A(n)>\mathcal F_X\e^{C-1}\}.
\end{align*}
Apply Theorem \ref{thm:main} with, for example, the admissible window
$Z(X)=L^{1/8}$.  Since the fixed value $z=C-1$ lies in that window for
all sufficiently large $X$,
\[
 \limsup_{j\to\infty}2^{-j}
 \#\bigl(E_C\cap[2^j,2^{j+1})\bigr)\ll\e^{-2C}.
\]
Lemma \ref{lem:dyadic-global} now gives
\[
 \overline d(E_C)\ll\e^{-2C}.
\]
Letting $C\to\infty$ proves that $\log A(n)-\log F(n)$ is tight in
natural density.  Adding the fixed constant
$\frac12\log2+\frac14\log c$ proves the displayed formulation of the
corollary.  Lemma \ref{lem:density-tight-equivalence} gives the equivalent
statement with every $\omega(n)\to\infty$.
\end{proof}

\begin{proof}[Proof of Theorem \ref{thm:no-equivalent}]
If such a dyadically regular $f$ existed, Corollary~\ref{cor:inherited-oscillation}
would give constants $0<\alpha<\beta$ such that, for every sufficiently large
$X$, there are $m,n\in[X,2X)$ with
\[
 f(m)\leq \alpha\mathcal F_X,
 \qquad
 f(n)\geq \beta\mathcal F_X.
\]
Thus
\[
 \sup_{X\leq u,v<2X}|\log f(u)-\log f(v)|\geq \log(\beta/\alpha)>0
\]
for all large $X$, contradicting dyadic regularity.  Hence no dyadically
regular asymptotic equivalent exists.
\end{proof}

\subsection{Resolution under dyadic regularity}
\label{subsec:resolution}

Theorem~\ref{thm:main} and Corollary~\ref{cor:global} determine the
second-order logarithmic scale of $A(n)$ in natural density.  Combined with
Theorem~\ref{thm:no-equivalent}, they settle the requested asymptotic
equivalence question within the dyadic-regularity formalization
of the word ``reasonable.''  In that explicit interpretation, no reasonable
deterministic function $f$ satisfies $A(n)\sim f(n)$ in natural density, and
$F$ is the best possible deterministic scale in the density-tight
sense
\[
 \log(A(n)/F(n))=O_{\mathrm{dens}}(1).
\]
This is a nonconcentration theorem rather than a failure to find the correct
multiplicative constant: the original question presupposes concentration, and
the mesoscopic tail bounds show that concentration does not occur.

This nonconcentration is a theorem-level feature, not a centering artifact.  Corollary~\ref{cor:inherited-oscillation} shows that any hypothetical asymptotic equivalent must inherit the dyadic spread of $A(n)$.  Ultimately the no-equivalent theorem uses Theorem~\ref{thm:main} only through the following two blockwise consequences:
there are fixed constants $0<a<b<1$ and $\delta>0$ such that, for all
sufficiently large $X$,
\[
 \PP_X(A(n)\leq a\mathcal F_X)\geq\delta,
 \qquad
 \PP_X(A(n)>b\mathcal F_X)\geq\frac34.
\]
Lemma~\ref{lem:no-scalar-center} makes the scalar obstruction explicit:
no scalar chosen on the dyadic block can capture $A(n)$ on almost all of that
block within a fixed multiplicative window.  Since the lemma holds for every
scalar $\lambda>0$, it also applies to any scalar selected from the block data,
including a median or any other one-number summary.  The issue is not poor
centering; it is nonconcentration.  Dyadic regularity then rules out
deterministic functions that attempt to evade this nonconcentration by
oscillating inside the dyadic block.

The caveats are sharper sequel questions, not ingredients needed for this
resolution.  Multiplying $F$ by a fixed constant merely translates the bounded
$z$-window, and Theorem~\ref{thm:no-equivalent} rules out all dyadically
regular deterministic normalizations at once.  The proof therefore does not
need a bounded-window limiting law or a limiting-law normalization constant.
It does not prove the Poisson prediction in \cite{Rafik26}, an exact second
factorial moment, or such a normalization constant.  Those remain natural
distributional problems beyond, and not required for, the dyadic-regularity result established here.

\section*{Acknowledgements}
We thank Zeraoulia Rafik for the 26 April 2026 online comment on Erd\H{o}s
Problem 731.  The comment publicly recorded the saddle-point centering and
Poisson heuristic discussed in the introduction, and it highlighted the
cross-base correlation issue.  The moving-base strip variance theorem in
this paper supplies the rigorous mesoscopic correlation estimate needed for
the tail bounds, density-tight scale, and nonconcentration theorem.  We
also thank the maintainers of the Erd\H{o}s Problems site for making the
problem record and its discussion thread publicly available.

\appendix

\section{Proof of the local-multiplicity large sieve}
\label{app:large-sieve}

We use the standard additive large sieve \cite{MV73}: if
$\theta_1,\ldots,\theta_J\in\T$ are $\delta$-separated, then
\begin{equation}
 \sum_{M<n\leq M+N}
 \abs{\sum_{j=1}^Jc_j\ee{n\theta_j}}^2
 \leq(N-1+\delta^{-1})\sum_{j=1}^J\abs{c_j}^2.         \label{eq:standard-ls}
\end{equation}

To prove Lemma \ref{lem:multiplicity-ls}, partition each half-circle
$[0,1/2)$ and $[1/2,1)$ into exactly $2X$ consecutive half-open intervals,
each of length $h=(4X)^{-1}$.  Every such interval is contained in an arc of
length $X^{-1}$ and therefore contains at most $\mathcal R$ indexed
frequencies.  Equal numerical frequencies are retained as separate indexed
elements and receive distinct ranks within their interval.  Rank the elements
from $1$ to $\mathcal R$, allowing unused ranks, and color an element by its
half-circle, its rank, and the interval index modulo $8$.  This uses at most
$16\mathcal R$ colors.

Two indexed frequencies of the same color lie in the same half-circle and in
intervals whose indices differ by at least $8$.  Their ordinary, hence
circular, distance is at least $7h>X^{-1}$; the use of two half-circles
prevents wrap-around at $0$.  Thus each color class is $X^{-1}$-separated.
Apply \eqref{eq:standard-ls} to each class with
$M=X-1$, $N=X$, and $\delta=X^{-1}$.  If the resulting exponential sums are
$F_1,\ldots,F_J$, with $J\leq16\mathcal R$, then
\[
 \abs{F_1+\cdots+F_J}^2\leq J\sum_{j=1}^J\abs{F_j}^2.
\]
Summing over $X\leq n<2X$ and dividing by $X$ proves
\eqref{eq:multiplicity-ls}.

\end{document}